\documentclass[a4paper,openbib,12pt]{article}
\usepackage{geometry}
\usepackage{amsmath}
\usepackage{amssymb}
\usepackage{amsfonts}

\newtheorem{theorem}{Theorem}[section]

\newtheorem{corollary}[theorem]{Corollary}

\newtheorem{definition}[theorem]{Definition}

\newtheorem{lemma}[theorem]{Lemma}

\newtheorem{problem}[theorem]{Problem}

\newenvironment{proof}[1][Proof]{\noindent\textbf{#1.} }{\ \rule{0.5em}{0.5em}}
\input{tcilatex}
\begin{document}

\title{Intersections of subcomplexes in non-positively curved 2-dimensional
complexes}
\author{Feng Ji and Shengkui Ye}
\maketitle

\begin{abstract}
Let $X$ be a contractible $2$-complex which is a union of two contractible
subcomplexes $Y$ and $Z.$ Is the intersection $Y\cap Z$ contractible as
well? In this note, we prove that the inclusion-induced map $\pi _{1}(Y\cap
Z)\rightarrow \pi _{1}(Z)$ is injective if $Y$ is $\pi _{1}$-injective
subcomplex in a locally CAT(0) 2-complex $X$. In particular, each component
in the intersection of two contractible subcomplexes in a CAT(0) 2-complex
is contractible.
\end{abstract}

\section{Introduction}

As a motivation, we consider the following problem.

\begin{problem}
\label{prob1}Let $X$ be a contractible $2$-complex which is a union of two
contractible subcomplexes $Y$ and $Z.$ Is the intersection $Y\cap Z$
contractible as well?
\end{problem}

A higher-dimensional version of this problem is already studied by Begle 
\cite{b}, which is related to the work of Aronszajn and Borsuk \cite{bb}.
Begle \cite{b} constructs a 3-dimensional contractible simplicial complex $%
X=Y\cup Z$ whose subcomplexes $Y,Z$ are both contractible but the
intersection $Y\cap Z$ is not simply connected. He left open the question as
to whether or not there are similar counter-examples in dimension two. We
study a more general problem as the following.

\begin{problem}
\label{prob}Let $X$ be a 2-dimensional aspherical simplicial complex (i.e.
the universal cover $\tilde{X}$ is contractible) and $Y$ be any $\pi _{1}$%
-injective subcomplex. For any subcomplex $Z\subset X,$ is the map $\pi
_{1}(Z\cap Y)\rightarrow \pi _{1}(Z)$ induced by the inclusion injective?
\end{problem}

Note that any subcomplex $Y$ of a contractible $2$-complex $X$ has vanishing
second homology group $H_{2}(Y;\mathbb{Z})=0$, by considering the long exact
sequence of homology groups for the pair $(X,Y)$. When $X$ is contractible
and $Z\subset X$ is also contractible, the triviality of $\pi _{1}(Z\cap Y)$
would imply that each connected component in $Z\cap Y$ is contractible by
the Whitehead theorem. This shows that a positive answer to Problem \ref%
{prob} gives a positive answer to Problem \ref{prob1}. We define a
subcomplex $Y$ of a 2-dimensional complex $X$ to be \emph{strongly }$\pi
_{1} $\emph{-injective} if for any subcomplex $Z$ of $X,$ the
inclusion-induced map $\pi _{1}(Y\cap Z)\rightarrow \pi _{1}(Z)$ is
injective (cf. Definition \ref{def}). The 2-complex $X$ is said to have
strong $\pi _{1}$-injectivity if any $\pi _{1}$-injective subcomplex $Y$ is 
\emph{strongly} $\pi _{1}$-injective.

We will give a positive answer to Problem \ref{prob} for locally $\mathrm{CAT%
}(0)$ 2-complexes by showing that locally $\mathrm{CAT}(0)$ 2-complexes have
strong $\pi _{1}$-injectivity, as the following.

\begin{theorem}
\label{th1.3}\bigskip Let $X$ be a proper nonpositively curved 2-complex and 
$Y$ a $\pi _{1}$-injective subcomplex. For any subcomplex $Z,$ the inclusion
induces an injection $\pi _{1}(Z\cap Y)\rightarrow \pi _{1}(Z).$ In other
words, $X$ has strong $\pi _{1}$-injectivity.
\end{theorem}

\begin{corollary}
Let $X$ be a $\mathrm{CAT}(0)$ 2-complex. For any two contractible
subcomplexes $Y,Z,$ each component in the intersection $Y\cap Z$ is
contractible.
\end{corollary}

Theorem \ref{th1.3} leads to the following observation: when $X$ is a finite
collapsible 2-complex and both $Y$ and $Z$ are contractible, each component
in the intersection $Z\cap Y$ is contractible (cf. Corollary \ref{lem3}).
This is already known by Segev \cite{se} using a different approach.

\textit{Notation: } All complexes are assumed to be connected simplicial
complexes, unless otherwise stated. We use $\pi _{1}(X)$ to denote the
fundamental group of $X$ with a based point in a connected component.

\section{Strong $\protect\pi _{1}$-injectivity}

We first give the following definition.

\begin{definition}
\label{def}A subcomplex $Y$ of a 2-dimensional complex $X$ is \emph{strongly 
}$\pi _{1}$\emph{-injective} if for any subcomplex $Z$ of $X,$ the
inclusion-induced map $\pi _{1}(Y\cap Z)\rightarrow \pi _{1}(Z)$ is
injective. The 2-complex $X$ has strong $\pi _{1}$-injectivity if any $\pi
_{1}$-injective subcomplex $Y$ is \emph{strongly} $\pi _{1}$-injective.
\end{definition}

Not every $2$-complex has strong $\pi _{1}$-injectivity. For a simple
counter-example, let $X$ be the sphere $S^2$. Since $S^{2}$ is a union of
two disks with the circle $S^{1}$ as the intersection, the upper disk in the
sphere $S^{2}$ is not strongly $\pi _{1}$-injective.

\begin{lemma}
\label{lemt}Any $\pi _{1}$-injective graph (i.e. $1$-simplicial subcomplex)
is strongly $\pi _{1}$-injective in any $2$-complex.
\end{lemma}

\begin{proof}
Let $X$ be a 2-complex with a $\pi _{1}$-injective $1$-dimensional
subcomplex $Y.$ Shrinking a contractible tree in $Y,$ we see that the
fundamental group of $Y$ is free. For a subcomplex $K\subset Y,$ the
fundamental group of $K$ is still free. If there is a non-nullhomotopic
closed loop in $K,$ the loop represents a nontrivial element in $Y.$ This
implies that the composite $\pi _{1}(K)\rightarrow \pi _{1}(Y)\rightarrow
\pi _{1}(X)$ is injective and thus $Y$ is strongly $\pi _{1}$-injective.
\end{proof}

Next, we study the relation between strong $\pi_1$-injectivity and taking
covering space.

\begin{lemma}
\label{lem1}Let $X$ be a 2-complex with a $\pi _{1}$-injective subcomplex $%
K. $ Suppose that $p:\tilde{X}\rightarrow X$ is the universal cover$.$ If $%
p^{-1}(K)$ is strongly $\pi _{1}$-injective in $\tilde{X},$ then the complex 
$K$ is strongly $\pi _{1}$-injective in $X.$
\end{lemma}

\begin{proof}
Suppose that there is a subcomplex $Z$ in $X$ such that $\pi _{1}(Z\cap
K)\rightarrow \pi _{1}(Z)$ is not injective. Let $f:S^{1}\rightarrow Z\cap K$
be a map whose homotopic class in $\pi _{1}(Z\cap K)$ is nontrivial, but
trivial in $\pi _{1}(Z).$ The map $f$ has a lifting $f^{\prime
}:S^{1}\rightarrow \tilde{X},$ since $[f]=1\in \pi _{1}(X).$ Moreover, $%
[f^{\prime }]\neq 1\in \pi _{1}(p^{-1}(K\cap Z),\ast )$ for the base point
in any connected component of $p^{-1}(K\cap Z).$ Since $\pi
_{1}(K)\rightarrow \pi _{1}(X)$ is injective, the complex (each connected
component) $p^{-1}(K)$ is simply connected. By assumption, the induced map $%
\pi _{1}(p^{-1}(K)\cap p^{-1}(Z))\rightarrow \pi _{1}(p^{-1}(Z))$ is
injective. Therefore, the homotopy class $[f^{\prime }]\neq 1\in \pi
_{1}(p^{-1}(Z)).$ This is a contradiction, since $\pi
_{1}(p^{-1}(Z))\rightarrow \pi _{1}(Z)$ is injective.
\end{proof}

Lemma \ref{lem1} implies that a 2-complex $X$ has strong $\pi _{1}$%
-injectivity if its universal cover $\tilde{X}$ does.

Let $X$ be a 2-complex and $K$ be a closed triangle (2-simplex). The
2-complex $X\cup K$ obtained by identifying two edges of $K$ with those of $%
X $ is called an elementary extension of $X,$ while $X$ is called an
elementary collapse of $X\cup K$ (cf. \cite{Bo}). Denote by $e$ the third
edge of $K,$ which is not in $X.$ A 2-complex $X$ is called collapsible if $%
X $ could be deformed to be a point by finite steps of elementary
extensions, collapses and contracting or adding free edges.

\begin{center}
\begin{equation*}
\FRAME{itbpF}{2.0928in}{1.9571in}{0in}{}{}{Figure}{\special{language
"Scientific Word";type "GRAPHIC";maintain-aspect-ratio TRUE;display
"USEDEF";valid_file "T";width 2.0928in;height 1.9571in;depth
0in;original-width 6.8476in;original-height 6.4031in;cropleft "0";croptop
"1";cropright "1";cropbottom "0";tempfilename
'PDMVFR00.bmp';tempfile-properties "XPR";}}
\end{equation*}

Figure 1. The elementary extension $X\cup K$
\end{center}

\begin{theorem}
\label{lem2}The 2-complex $X$ has strong $\pi _{1}$-injectivity if and only
if so does an elementary extension $X\cup K.$
\end{theorem}

\begin{proof}
Suppose that the elementary extension $X\cup K$ has strong $\pi _{1}$%
-injectivity. For any subcomplex $Y$ with injective fundamental group, we
see that 
\begin{equation*}
\pi _{1}(Y)\rightarrow \pi _{1}(X)\overset{\cong }{\rightarrow }\pi
_{1}(X\cup K)
\end{equation*}%
is injective as well. Therefore, for any subcomplex $Z\subset X,$ the map $%
\pi _{1}(Z\cap Y)\rightarrow \pi _{1}(Z)$ is injective.

Conversely, suppose that $X$ has strong $\pi _{1}$-injectivity. Let $%
Y\subset X\cup K$ be any $\pi _{1}$-injective subcomplex and $Z\subset X\cup
K$ any subcomplex. We divide the proof into several cases.

\begin{enumerate}
\item[Case 1] $Y\supset K.$

\item[1.1] $Z\supset K.$ For convenience, let $Y\backslash K$ denote the
subcomplex of $Y$ obtained by deleting the interior of $K$ and the open edge 
$e.$ We see that 
\begin{equation*}
Y\cap Z=(Y\backslash K\cap Z\backslash K)\cup K.
\end{equation*}%
Note that $Y$ is an elementary extension of $Y\backslash K$ and $Y\cap Z$ is
also an elementary extension of $Y\backslash K\cap Z\backslash K.$
Therefore, we get by the hypothesis on $X$ that 
\begin{equation*}
\pi _{1}(Y\cap Z)=\pi _{1}(Y\backslash K\cap Z\backslash K)\hookrightarrow
\pi _{1}(Z\backslash K)=\pi _{1}(Z).
\end{equation*}

\item[1.2] $Z\nsupseteqq K$ but $Z\supset e.$ Let $Z\backslash e$ denote the
complex obtained by removing the interior of $e$ from $Z$. We have that 
\begin{equation*}
Y\cap Z=(Y\backslash K\cap Z\backslash e)\cup e.
\end{equation*}%
If two ends of $e$ are both in the same component of $Y\backslash K\cap
Z\backslash e,$ let $P$ be a path in $Y\backslash K\cap Z\backslash e$
connecting the two ends. Choose a base point $x_{0}\in P.$ Contracting the
path $P,$ we have that (note the injectivity of the first free factor
follows from the hypothesis on $X$) 
\begin{equation*}
\pi _{1}(Y\cap Z,x_{0})=\pi _{1}(Y\backslash K\cap Z\backslash e,x_{0})\ast 
\mathbb{Z}\hookrightarrow \pi _{1}(Z\backslash e,x_{0})\ast \mathbb{Z}=\pi
_{1}(Z,x_{0}).
\end{equation*}%
If the two ends of $e$ lie in two different components $Y_{1},Y_{2}$ of $%
Y\backslash K\cap Z\backslash e,$ choose the path $F$ consisting of the two
attaching edges in $K.$ Note that $F$ is not in $Z$. Since $X$ has strong $%
\pi _{1}$-injectivity, there is an injection 
\begin{equation*}
\pi _{1}(Y\backslash K\cap (Z\backslash e\cup F),x_{0})\hookrightarrow \pi
_{1}(Z\backslash e\cup F,x_{0})
\end{equation*}%
where the base point $x_{0}\ $is one end of $e.$ Therefore, we have that%
\begin{eqnarray*}
\pi _{1}(Y\cap Z,x_{0}) &=&\pi _{1}(Y_{1})\ast \pi _{1}(Y_{2})=\pi
_{1}((Y\backslash K\cap Z\backslash e)\cup F,x_{0}) \\
&\hookrightarrow &\pi _{1}(Z\backslash e\cup F,x_{0})=\pi _{1}(Z,x_{0}).
\end{eqnarray*}

\item[1.3] $Z\subset X.$ We have that $Y\cap Z=Y\backslash K\cap Z$ and thus%
\begin{equation*}
\pi _{1}(Y\cap Z)=\pi _{1}(Y\backslash K\cap Z)\hookrightarrow \pi _{1}(Z).
\end{equation*}

\item[Case 2] $Y\nsupseteqq K$ but $Y\supset e,$ where $e$ is the closed
edge of $K$ not in $X.$

\item[2.1] $Z\supset K.$ We have that $Y\cap Z=(Y\backslash e\cap
Z\backslash K)\cup e.$ Since $\pi _{1}(Y)\rightarrow \pi _{1}(X\cup K)$ is
injective, the path $F$ consisting of the two attaching edges of $K$ does
not lie in $Y.$ If the two ends of $e$ lie in the same component of $%
Y\backslash e\cap Z\backslash K,$ the edge $e$ is a part of a loop in $Y\cap
Z.$ Then the path $F$ is part of a loop in $Y\backslash e\cap Z\backslash K$
by replacing $e$ with $F$. If the two ends of $e$ lie in different
components, then the edge $e$ will not contribute to the fundamental group.
In any case, we have an injection%
\begin{equation*}
\pi _{1}(Y\cap Z)\hookrightarrow \pi _{1}((Y\backslash e\cup F)\cap
Z\backslash K).
\end{equation*}%
Since $\pi _{1}(Y)=\pi _{1}(Y\backslash e\cup F)\hookrightarrow \pi
_{1}(X\cup K)=\pi _{1}(X),$ the subcomplex $Y\backslash e\cup F$ is also $%
\pi _{1}$-injective. Considering that $X$ has strong $\pi _{1}$-injectivity,
there is an injection 
\begin{equation*}
\pi _{1}((Y\backslash e\cup F)\cap Z\backslash K)\hookrightarrow \pi
_{1}(Z\backslash K)=\pi _{1}(Z).
\end{equation*}%
This proves that the inclusion induces an injection $\pi _{1}(Y\cap
Z)\hookrightarrow \pi _{1}(Z).$

\item[2.2] $Z\nsupseteqq K$ but $Z\supset e.$ Since $\pi _{1}(Y)\rightarrow
\pi _{1}(X\cup K)$ is injective, the path $F$ consisting of the two
attaching edges of $K$ does not lie in $Y.$ For the same reason as that of
the case 2.1, we have an injection%
\begin{eqnarray*}
\pi _{1}(Y\cap Z) &=&\pi _{1}((Y\backslash e\cap Z\backslash e)\cup
e)\hookrightarrow \pi _{1}((Y\backslash e\cap Z\backslash e)\cup F) \\
&=&\pi _{1}((Y\backslash e\cup F)\cap (Z\backslash e\cup F)).
\end{eqnarray*}%
Note that $\pi _{1}(Y\backslash e\cup F)=\pi _{1}(Y)\hookrightarrow \pi
_{1}(X\cup K)=\pi _{1}(X).$ Since $X$ has strong $\pi _{1}$-injectivity, the
inclusion induces an injection $\pi _{1}((Y\backslash e\cup F)\cap
(Z\backslash e\cup F))\hookrightarrow \pi _{1}(Z\backslash e\cup F).$ If $%
F\nsubseteqq Z,$ we have that $\pi _{1}(Z\backslash e\cup F)=\pi _{1}(Z).$
If $F\subset Z,$ we have that $\pi _{1}(Z\backslash e\cup F)\ast \mathbb{Z}%
=\pi _{1}(Z).$ In both cases, there is an injection $\pi _{1}(Z\backslash
e\cup F)\hookrightarrow \pi _{1}(Z).$ Therefore, the map $\pi _{1}(Y\cap
Z)\rightarrow \pi _{1}(Z)$ is injective.

\item[2.3] $Z\subset X.$ We have that $Y\cap Z=Y\backslash e\cap Z.$ For the
same reason as that of the case 2.1, the path $F$ is not in $Y$ and there is
an injection $\pi _{1}(Y\backslash e\cup F)\hookrightarrow \pi _{1}(X).$ We
have that%
\begin{equation*}
\pi _{1}(Y\cap Z)=\pi _{1}(Y\backslash e\cap Z)\hookrightarrow \pi
_{1}((Y\backslash e\cup F)\cap Z)\hookrightarrow \pi _{1}(Z).
\end{equation*}

\item[Case 3] $Y\subset X.$

\item[3.1] $Z\supset K.$ We have that $Y\cap Z=Y\cap (Z\backslash K).$ Since 
$Z\backslash K$ is a collapse of $Z,$ we get from the hypothesis on $X$ that 
\begin{equation*}
\pi _{1}(Y\cap Z)=\pi _{1}(Y\cap (Z\backslash K))\hookrightarrow \pi
_{1}(Z\backslash K)=\pi _{1}(Z).
\end{equation*}

\item[3.2] $Z\nsupseteqq K$ but $Z\supset e.$ In this case, $\pi _{1}(Z)=\pi
_{1}(Z\backslash e)\ast \mathbb{Z}.$ The hypothesis on $X$ implies that $\pi
_{1}(Y\cap Z)=\pi _{1}(Y\cap Z\backslash e)$ injects into $\pi
_{1}(Z\backslash e)$. Therefore, we have an injection 
\begin{equation*}
\pi _{1}(Y\cap Z)\hookrightarrow \pi _{1}(Z\backslash e)\hookrightarrow \pi
(Z\backslash e)\ast \mathbb{Z}=\pi _{1}(Z).
\end{equation*}

\item[3.3] $Z\subset X.$ This subcase follows directly from the hypothesis
of $X$.

All the cases are included and the proof is complete.
\end{enumerate}
\end{proof}

It is already known by Segev \cite{se} (4.3) that when $X$ is a finite
collapsible 2-complex and both $Y$ and $Z$ are contractible, each connected
component in the intersection $Z\cap Y$ is contractible. This is a special
case of the following.

\begin{corollary}
\label{lem3}A collapsible 2-complex has strong $\pi _{1}$-injectivity. In
particular, each connected component in the intersection $Y\cap Z$ of two
contractible subcomplexes $Y,Z$ in a collapsible 2-complex $X$ is
contractible.
\end{corollary}

\begin{proof}
A collapsible 2-complex is deformed to a point by a finitely many elementary
collapse or extensions. The first part is thus implied by Theorem \ref{lem2}%
. When $Y$ and $Z$ are contractible, the intersection $Y\cap Z$ is $\pi _{1}$%
-injective in $Z$ and thus simply connected. A simply connected subcomplex
of a contractible 2-complex is acyclic by the relative homology exact
sequence. Therefore, each connected component in the intersection $Y\cap Z$
is contractible by the Whitehead theorem.
\end{proof}

\section{Non-positively curved complexes}

Recall the notion of non-positively curved complexes from Bridson and
Haefliger \cite{bh} (Chapter II. 1.2). Let $(X,d)$ be a geodesic metric
space. A geodesic triangle $\Delta (x,y,z)$ consists of three vertices $%
x,y,z\in X$ and three geodesics $[x,y],[y,z],[x,z]$ connecting these
vertices. A comparison triangle $\Delta (\bar{x},\bar{y},\bar{z})$ (or
denoted by $\bar{\Delta}(x,y,z)$) is an Euclidean triangle in the plane $%
\mathbb{R}^{2}$ with three vertices $\bar{x},\bar{y},\bar{z}$ and edges of
lengths $d(x,y),d(y,z),d(x,z)$ respectively.

\begin{definition}
\label{cat}A geodesic metric space $X$ is CAT(0) if for any geodesic
triangle $\Delta (x,y,z)$ and any two points $p,q\in \Delta (x,y,z),$ we have%
\begin{equation*}
d(p,q)\leq d_{\mathbb{R}^{2}}(\bar{p},\bar{q}),
\end{equation*}%
where $\bar{p},\bar{q}$ are the corresponding points of $p,q$ in the
comparison triangle $\Delta (\bar{x},\bar{y},\bar{z}).$
\end{definition}

A Euclidean cell is the convex hull of a finite number of points in $\mathbb{%
R}^{n}$, equipped with the standard Euclidean metric. A Euclidean cell
complex $X$ is a space formed by gluing together Euclidean cell-complexes
via isometries of their faces. It has the piecewise Euclidean path metric.
Precisely, for any $x,y\in X,$ let $x=x_{0},x_{1},\cdots ,x_{n}=y$ be a path
such that each successive $x_{i},x_{i+1}$ is contained in a Euclidean
simplex $S_{i}.$ Define the distance (called path metric) $d_{X}(x,y)=\inf
\sum_{i=0}^{n-1}d_{S_{i}}(x_{i},x_{i+1}),$ where the infimum is taken over
all such paths. Note that a metric space $X$ is proper if any closed ball $%
B(x,r)\subset X$ is compact.

\begin{definition}
A Euclidean cell complex $X$ is non-positively curved if it is locally
CAT(0), i.e. for every $x\in X$ there exists $r_{x}>0$ such that the ball $%
B(x,r_{x})$ with the induced metric is a CAT(0).
\end{definition}

Let $X$ be a Euclidean cell complex and $v\in X.$ The (geometric) link $%
Lk(x,X)$ is the set of unit tangent vectors ("directions") at $x$ in $X.$
Precisely, let $S$ be the set of all geodesics $[x,y]$ with $y$ in a simplex
containing $x.$ Two geodesics are called equivalent if one is contained in
the other. The link $Lk(x,X)$ is the set of equivalence classes of geodesics
in $S.$ If $X$ is one $n$-dimensional Euclidean cell, the link $Lk(x,X)$ is
part of $S^{n-1}$ and thus the topology on $Lk(x,X)$ is defined as the
"angle" topology. In general, the topology on $Lk(x,X)$ is defined as the
path metric coming from each cell.

We will need the following facts about $2$-complexes from \cite{bh}.

\begin{lemma}
\label{key}(1) A finite $\mathrm{CAT}(0)$ 2-complex is collapsible.

(2) (Link condition) A $2$-dimensional Euclidean cell complex $X$ is
non-positively curved if and only if each link $Lk(x,X)$ contains no
injective loops of length less than $2\pi .$

(3) A simply connected non-positively curved complex is $\mathrm{CAT}(0).$
\end{lemma}

\begin{proof}
The first claim (1) is \cite{bh}, 5.34(2), while the second claim (2) is 
\cite{bh}, 5.5 and 5.6 in Chapter II.5. The last claim follows the
Cartan-Hadamard theorem (see \cite{bh}, II. 4.1).
\end{proof}

\bigskip

\begin{proof}[Proof of Theorem \protect\ref{th1.3}]
Since the universal cover of a non-positively curved complex is $\mathrm{CAT}%
(0),$ it suffices to prove that a $\mathrm{CAT}(0)$ 2-complex $X$ has strong 
$\pi _{1}$-injectivity by Lemma \ref{lem1}. By Corollary \ref{lem3}, any
collapsible 2-complex has strong $\pi _{1}$-injectivity. As a finite $%
\mathrm{CAT}(0)$ 2-complex is collapsible (see Lemma \ref{key} (1)), it has
strong $\pi _{1}$-injectivity. Suppose that a simply connected subcomplex $%
Y\subset X$ is not strongly $\pi _{1}$-injective. Let $Z\subset X$ be a
subcomplex such that $\pi _{1}(Y\cap Z)\overset{i}{\rightarrow }\pi _{1}(Z)$
is not injective. Choose $f:S^{1}\rightarrow Y\cap Z$ such that the homotopy
class $[f]\in \ker i$ is not trivial. Since $\text{Im}f$ is compact and any
homotopy $h$ between $f$ and a constant map has compact image in $Z,$ we
could choose finite subcomplexes $Y^{\prime }\subset Y$ containing $\text{Im}%
f$ and $Z^{\prime }\subset Z$ containing $\text{Im}h.$ The link condition
implies that the subcomplex $Y$ is non-positively curved (cf. Lemma \ref{key}
(2)). Since $Y$ is simply connected, the subcomplex $Y$ is $\mathrm{CAT}(0)$
by Lemma \ref{key} (3). The finite subcomplex $Y^{\prime }$ is contained in
a ball $B_{Y}(x,r)\subset Y$ of sufficiently large radius, for some point $%
x\in Y$ and sufficient large $r$. When $X$ is proper, the closed ball $%
B_{Y}(x,r)$ is compact. Since the ball $B_{Y}(x,r)$ is contractible (cf. 
\cite{bh}, II.1.4), we may choose a finite contractible subcomplex $%
Y^{\prime \prime }\subset Y$ containing $Y^{\prime }$ (for example, take $%
Y^{\prime \prime }=B_{Y}(x,r)$). By the construction, the inclusion-induced
map $\pi _{1}(Y^{\prime \prime }\cap Z^{\prime })\overset{i}{\rightarrow }%
\pi _{1}(Z^{\prime })$ is not injective. Since both $Y^{\prime \prime }$ and 
$Z^{\prime }$ are finite, we may choose a ball $B_{X}(x,r^{\prime })$ of
sufficiently large radius containing $Y^{\prime \prime }$ and $Z^{\prime }.$
Therefore, the subcomplexes $Y^{\prime \prime }$ and $Z^{\prime }$ is
contained in a finite $\mathrm{CAT}(0)$ 2-complex $X^{\prime }.$ The strong $%
\pi _{1}$-injectivity of $X^{\prime }$ implies that the map $\pi
_{1}(Y^{\prime \prime }\cap Z^{\prime })\overset{i}{\rightarrow }\pi
_{1}(Z^{\prime })$ is injective, which gives a contradiction. This finishes
the proof.
\end{proof}

\bigskip

Infinitus, Nanyang Technological University, 50 Nanyang Ave, S2B4b05,

Singapore 639798. E-mail: jifeng@ntu.edu.sg

\bigskip

Department of Mathematical Sciences, Xi'an Jiaotong-Liverpool University,
111 Ren Ai Road, Suzhou, Jiangsu, China 215123. E-mail:
Shengkui.Ye@xjtlu.edu.cn

\end{document}